\documentclass[12pt]{amsart}
\usepackage{amsfonts, amsmath, amssymb,amsthm}
\usepackage{amsfonts}
\usepackage{enumitem}
\usepackage[a4paper]{geometry}
\usepackage{mathtools}

\usepackage{algorithm}
\usepackage{algorithmic}

\usepackage{color}

\newtheorem{theorem}{Theorem}[section]

\newtheorem{proposition}[theorem]{Proposition}

\theoremstyle{definition}
\newtheorem{example}[theorem]{Example}
\newtheorem{examples}[theorem]{Examples}
\newtheorem{definition}[theorem]{Definition}
\newtheorem{observation}[theorem]{Observation}

\newtheorem{remark}[theorem]{Remark}
\newtheorem{remarks}[theorem]{Remarks}

\newcommand{\half}{\frac{1}{2}}

\newcommand{\TT}{\mathbb{T}}
\newcommand{\TP}{\mathbb{TP}}
\newcommand{\RR}{\mathbb{R}}
\newcommand{\ZZ}{\mathbb{Z}}
\newcommand{\Projn}{\Ran / \RR \One}
\newcommand{\Rn}{\RR^{n}}
\newcommand{\Zn}{\ZZ^{n}}
\newcommand{\Ran}{\RR^{n+1}}

\newcommand{\ta}{\oplus}
\newcommand{\tm}{\odot}

\newcommand{\varx}{{\bf x}}
\newcommand{\vary}{{\bf y}}
\newcommand{\varz}{{\bf z}}

\newcommand{\varp}{{\bf p}}
\newcommand{\varq}{{\bf q}}
\newcommand{\One}{{\bf 1}}
\newcommand{\Zero}{{\bf 0}}

\newcommand{\tnorm}[1]{\parallel #1 \parallel_{\mathrm{tr}} \,}

\newcommand{\Bn}{B_{\mathrm{tr}}^n}

\newcommand{\BRn}{B_{\mathrm{tr},R}^n}

\newcommand{\San}{S_{\mathrm{tr}}^{n-1}}

\newcommand{\In}[1]{I_{#1}^n}

\newcommand{\abs}[1]{\lvert #1 \rvert}

\numberwithin{equation}{section}

\usepackage{pgfplots}
\pgfplotsset{compat = newest}
\usepgfplotslibrary{fillbetween}

\usepackage{calc}
\usepackage{tikz}
\usetikzlibrary{decorations.markings}
\tikzstyle{vertex}=[circle, draw, inner sep=0pt, minimum size=6pt]

\usepackage{graphicx}
\graphicspath{ {./images/} }
\usetikzlibrary{shapes.geometric}
\usetikzlibrary{shapes.geometric}
\usetikzlibrary{arrows}

\begin{document}
\title[]{Tropical balls, geodesics and honeycomb}
\author{Amnon Rosenmann}
\email[]{rosenmann@math.tugraz.at}
\date{}

\begin{abstract}
In these notes we describe the geometry of tropical balls in $\mathbb{R}^n$ equipped with the tropical metric. 
After defining the tropical length of rectifiable curves (and not just piecewise linear curves), 
we characterize compact tropically geodesic sets in $\mathbb{R}^n$.
Next, we describe the tropical unit ball as a zonotope, via its tropical generating set, as a union of $n+1$ tropical unit hypercubes, and as the tropical geodesic hull of the tropical unit vectors.
Finally, we show that translates of the tropical unit ball whose centers lie in a sublattice of $\mathbb{Z}^n$ form a facet-to-facet honeycomb tiling of $\mathbb{R}^n$.
We note that a great part of the material presented here is either known or implied from known results.
\end{abstract}
\maketitle
\tableofcontents{}

\section{Introduction}
In \cite{CGQ04} (see also \cite{Pue14}) Cohen, Gaubert and Quadrat introduced a tropical metric that is an additive version of Hilbert projective metric.
This metric is direction-dependent: the tropical distance between two points is invariant to translations but not to rotations.
After recalling the tropical metric, we define the tropical length of a rectifiable curve by way of a standard partition. 
This allows us to define tropical geodesics among rectifiable curves and not only among piecewise linear curves.
Next, we characterize compact tropically geodesic sets in $\Rn$.
They are precisely the compact polytopes that are intersections of half-spaces of the forms
$a_i \le x_i \le a_i'$ and $x_i-x_j \ge b_{ij}$.
This description makes clear that there are only finitely many combinatorial types of compact tropically geodesic sets in each dimension, where the different types depend on which of these inequalities are redundant.

A central object throughout these notes is the $n$-dimensional tropical unit ball $\Bn$. Its definition with respect to the tropical metric results in a polyhedral shape that is stretched along the line $\lambda \One$.
We describe $\Bn$ from several complementary viewpoints. 
First, it is a zonotope: it can be realized as the Minkowski sum of the tropical unit segments in the $n+1$ tropical coordinate directions, which immediately yields a polyhedral description of its faces. 
Secondly, $\Bn$ decomposes into $n+1$ tropical unit hypercubes, disjoint up to common facets, a decomposition that is induced by a natural tropical coordinate system with $n+1$ orthants. 
Thirdly, $\Bn$ is the tropical geodesic hull of the tropical unit vectors.
Finally, it is generated by the set $S$ of the (negations) of these vectors, i.e., it is the tropical convex hull of $S$, as introduced by Develin and Sturmfels \cite{DS04}. 
As an additional geometric notion, we define the tropical angle between two standard lines via the intrinsic metric on the tropical unit sphere.

The last part of the paper is about a constructive honeycomb theorem for tropical balls. 
We show that translates of $\Bn$ with centers at the sublattice
$$
	\{ {\bf c} = (c_1,\ldots,c_n) \in\Zn : \sum_{i=1}^{n} c_i \equiv 0 \: (\mathrm{mod}~n+1)
$$
tile $\Rn$ facet-to-facet. 
While this tiling property is consistent with general results on space-tiling zonotopes (by  McMullen \cite{McM75}), we give an explicit proof that assigns to each point $\varx \in \Rn$ (away from boundaries) a unique ball $\Bn ({\bf c})$ to which $\varx$ belongs.

\medskip
\noindent\textbf{Organization.}
Section~\ref{sec:proj} reviews the tropical projective torus and fixes notation. 
Section~\ref{sec:metric} recalls the tropical metric and its basic properties. 
In Section~\ref{sec:geodesics} we define the tropical length of rectifiable curves, introduce tropical geodesics, and characterize compact tropically geodesic sets.
In Section~\ref{sec:ball} we present several equivalent descriptions of the tropical unit ball, that can serve as definitions.
Finally, Section~\ref{sec:honeycomb_Rn} is about the tiling of $\mathbb{R}^n$ by translates of the tropical unit ball.

\indent
 A great part of the material presented here is either known or implied from known results.

\section{The tropical projective torus}
\label{sec:proj}
Tropical geometry may be seen as a degeneration of complex structures into piecewise linear and polyhedral structures through Maslov's log-limit ``dequantization'' (see \cite{KM97}).
It studies polynomials and varieties over the semifield of extended real numbers with $+\infty$ (or $-\infty$) and the idempotent operation of minimum (or maximum) in place of addition and the operation of addition in place of multiplication.
The result is a kind of linear and combinatorial and, in a way, simpler version of algebraic geometry. Applications are, among others, in algebraic geometry, optimization problems, discrete event dynamical systems, neural networks and string theory. For the main source on tropical geometry we refer to the book of Maclagan and Sturmfels \cite{MS15}; other selected sources are \cite{IMS09}, \cite{Alg13}, \cite{MR}, \cite{PS05}, \cite{Gro15},\cite{Jos21}. 
Here we restrict ourselves to the tropical projective torus equipped with the tropical metric. We start by reviewing the basic notions of tropical algebra.

The (min) {\it tropical semifield} (or {\it min-plus algebra}) ($\TT,\ta, \tm$) consists of the set $\TT = \RR \cup {\infty}$ equipped with the operations of tropical addition $\ta$ and tropical multiplication $\tm$, defined by
\begin{equation}
	a \ta b := \min(a,b) , \qquad a \tm b := a+b.
\end{equation}
The identity element for addition is $\infty$ and the identity element for multiplication is $0$.
The operations of addition and multiplication are associative and commutative, and multiplication is distributive over addition.
Tropical algebra is part of ``idempotent mathematics'', which was developed mainly by Maslov and his collaborators (see \cite{Lit07} for a brief introduction). 
Analogous is the {\it max-plus algebra} (see the monograph of Butkovi{\v{c}} \cite{But10}), in which the maximum operation plays the role of addition and the identity element for addition is $-\infty$.
We choose here to work mainly over the min-plus algebra but also mention in most cases what happens over the isomorphic semifield of max-plus algebra.

The construction of the $n$-dimensional semimodule $\TT^n$ over $\TT$ is standard. Given $\varx = (x_1,\ldots,x_{n}), {\it \vary} = (y_1,\ldots,y_{n}) \in \TT^n$ and $a \in \TT$, vector addition is defined entry-wise by $\varx \ta \vary = (x_1 \ta y_1, \ldots, x_n \ta y_n)$ and multiplication by a scalar is defined by $a \tm \varx =  (a \tm x_1, \ldots, a \tm x_n)$.

Since we are only interested here in bounded objects, we work over $\TT^{\times} = \TT~\setminus~\{\infty\} = \RR$.
We then look at the $n$-dimensional {\it tropical projective torus} $\Projn$ (sometimes denoted $\TP^{n}$),
where $\One = (1,\ldots,1)$, that is, for all $a \in \RR$, $(x_1,\ldots,x_{n+1}) \sim (x_1+a,\ldots,x_{n+1}+a)$, so that the elements of $\Projn$ are the parallel lines in $\Ran$ in direction $\One$.
For each element of $\Projn$, written in homogeneous coordinates as $(x_1 : \ldots : x_{n+1})$, we can choose as a representative the element $(x_1 - x_{n+1} : \ldots : x_n - x_{n+1}:0)$, and then map it bijectively to the element $(x_1 - x_{n+1}, \ldots, x_n - x_{n+1})$ of $\Rn$. 
In what follows, when we use the coordinates $(x_1,\ldots,x_{n})$ we refer to this representation of $(x_1 : \ldots : x_n : 0) \in \Projn$ as an element of $\Rn$.

\section{Tropical metric}
\label{sec:metric}
The {\it (min) tropical line segment} (see \cite{MS15}, p. 230) between the points $\varx=(x_1,\ldots,x_n)$ and $\vary=(y_1,\ldots,y_n)$ of $\Rn$
is a piecewise linear curve. It contains a minimal vertex $\varz=\varx \ta \vary = (\min(x_1,y_1),\ldots,\min(x_n,y_n))$ (where $\varz$ may be $\varx$ or $\vary$) and two positive-oriented piecewise linear curves, one from $\varz$ to $\varx$ and one from $\varz$ to $\vary$, where each line segment is in direction of a vector with zeros and ones. The set of vertices along the tropical line segment from $\varz$ to $\varx$ 
(where some vertices may coincide) is
\begin{equation}
	\{(x_j-z_j) \tm \varz \ta \varx : j = 1, \ldots, n\}.
\end{equation}
The vertices along the path from $\varz$ to $\varx$ may be described in the following way.
We look at the entries of $\varz$ as representing the initial values of $n$ clocks $c_1,\ldots,c_n$. The clocks then start to run simultaneously and at the same rate. At each vertex along the path at least one clock, say $c_i$, stops running after reaching its target value $x_i$. The rest of the clocks continue to run at the same rate until the next vertex, where at least one more clock stops running, and so on until every clock $c_j$ reaches its final value $x_j$.
The vertices along the tropical line segment from $\varz$ to $\vary$ are defined similarly, and overall there are at most $n$ standard line segments between $\varx$ and $\vary$.

The above description suggests a metric on $\Rn$ in which the distance between 
$\varx$ and $\vary$ is the sum of the passing time when traveling from $\varz$ to $\varx$ and the time from $\varz$ to $\vary$. This is indeed the definition of the tropical distance (see \cite{CGQ04}, \cite{Pue14}):
\begin{definition}
	\label{def:dist}
	The {\it tropical distance} in $\Projn$ between $\varx = (x_1 : \ldots : x_{n+1})$ and $\vary = (y_1 : \ldots : y_{n+1})$ is
	\begin{equation}
		\label{eq:dist}
		d_{\mathrm{tr}}(\varx, \vary) := \max_{1 \leq i,j \leq n+1} \{x_i - y_i -x_j + y_j\}.
	\end{equation}
	\end{definition}
	When identifying $\Projn$ with $\Rn$, as done above, the tropical distance between $\varx=(x_1,\ldots,x_n)$ and $\vary=(y_1,\ldots,y_n)$ is
	\begin{eqnarray}
		d_{\mathrm{tr}}(\varx, \vary) &:=& \max\{\max_{1 \leq i \leq n} \{\abs{x_i - y_i}\}, \max_{1 \leq i,j \leq n} \{x_i - y_i -x_j + y_j\}\} \\
		&=& \max\{\max_{1 \leq i \leq n} \{x_i - y_i\},0\}
		- \min\{ \min_{1 \leq i \leq n} \{ x_i - y_i\}, 0\}. \nonumber
	\end{eqnarray}
\begin{remark}
	When representing the elements of $\Rn$ using $n+1$ coordinates, as done in Subsection~\ref{subsec:coord}, then the distance is according to formula~\eqref{eq:dist}.
\end{remark} 

The tropical distance defines a {\it tropical metric} on $\Projn$ and $\Rn$ that is positioned between $l^1$ metric (Manhattan metric) and $l^\infty$ metric (Chebyshev metric).
The distance between two points in those three metrics is measured with respect to shortest paths that are piecewise linear with non-degenerate line segments of directions $(x_1,\ldots,x_n)$ that satisfy the following conditions.
In $l^1$ metric the line segments are parallel to the coordinate axes, that is, all $x_i$ are $0$ except for one $x_j \in \{1,-1\}$.
In the tropical metric all $x_i \in \{0,1\}$ or all $x_i \in \{0,-1\}$.
In $l^\infty$ metric all $x_i \in \{0,1,-1\}$. The above direction vectors $(x_1,\ldots,x_n)$ are all unit vectors in the relevant metrics. Hence, we have the following bi-Lipschitz comparison:
$$d_{\infty}(\varx, \vary) \leq d_{\mathrm{tr}}(\varx, \vary) \leq 2d_{\infty}(\varx, \vary)$$
and
$$d_{\mathrm{tr}}(\varx, \vary) \leq d_{1}(\varx, \vary) \leq nd_{\mathrm{tr}}(\varx, \vary).$$

The tropical norm of a vector is its distance from the origin.
\begin{definition}
	The {\it tropical norm} of $\varx = (x_1 : \ldots : x_{n+1}) \in \Projn$ is
	\begin{equation}
		\tnorm{\varx} := \max_{1 \leq i, j \leq n+1} \{x_i - x_j\}.
	\end{equation}
	In $\Rn$ the tropical norm of $\varx=(x_1,\ldots,x_n)$ is
	\begin{equation}
	 	\tnorm{\varx} := \max\{\max_{1 \leq i \leq n} \{x_i\}, 0\} - 
		\min \{\min_{1 \leq i \leq n} \{x_i\}, 0\}.
	\end{equation}
\end{definition}

For example, $\tnorm{(-3,-2,1)} = 4$ and $\tnorm{(-3,-2,-1)} = 3$, whereas \\
$\tnorm{(-3 : -2 :-1)} = 2$.
\begin{remark}
	The tropical line segment between $\varx$ and $\vary$ in the max-plus setting is defined in a similar way to the one in the min-plus setting. It passes through the vertex $\varz$ with entries $z_i = \max(x_i,y_i)$ and the tropical line segments from $\varz$ to $\varx$ and from $\varz$ to $\vary$ consist of (standard) line segments in directions of vectors with entries in $\{0,-1\}$.
	The defined metric is exactly the same in both settings. 
\end{remark}
\section{Tropical geodesics}
\label{sec:geodesics}
Part of the results in this section can be found in the book of Joswig \cite{Jos21}.

Having the tropical metric at hand, we are able to measure the tropical length of any piecewise linear line which decomposes into finitely many tropical line segments.
Next, we extend this notion of length in a standard way to more general curves.
\begin{definition}
	\label{def:length}
	Let $\gamma : [a,b] \to \Rn$ be a finite parametrically-defined rectifiable curve (that is, the mapping $\gamma$ is continuous, Lipschitz and almost everywhere differentiable).
	Let
	\begin{equation}
		l_{\mathrm{tr}, \varepsilon}(\gamma) := \inf_{\mathcal{T}_{\alpha}=(t_{\alpha,i})_{i=0}^{n_\alpha}} \left\{ \sum_{i=1}^{n_{\alpha}} d_{\mathrm{tr}}(\gamma(t_{\alpha,i-1}), \gamma(t_{\alpha,i})) \right\},
	\end{equation}
	where the infimum is over all partitions $\mathcal{T}_{\alpha}=(t_{\alpha,i})_{i=0}^{n_\alpha}$ of $[a,b]$ satisfying $a=t_{\alpha,0} < t_{\alpha,1} < \cdots < t_{\alpha,n_{\alpha}}=b$, with $d_{\mathrm{tr}}(\gamma(t_{\alpha,i-1}), \gamma(t_{\alpha,i})) \leq \varepsilon$, for all $i$. Then the {\it tropical length} of $\gamma$ is
	\begin{equation}
		l_{\mathrm{tr}}(\gamma) := \lim_{\varepsilon \to 0^+} l_{\mathrm{tr},\varepsilon}(\gamma).
	\end{equation}
\end{definition}
\begin{remark}
	An equivalent definition of the tropical length of a rectifiable curve $\gamma : [a,b] \to \Rn$ is by its total variation:
	\begin{equation}
		l_{\mathrm{tr}}(\gamma) := \sup_{\mathcal{T}_{\alpha}=(t_{\alpha,i})_{i=0}^{n_\alpha}} \left\{ \sum_{i=1}^{n_{\alpha}} d_{\mathrm{tr}}(\gamma(t_{\alpha,i-1}), \gamma(t_{\alpha,i})) \right\},
	\end{equation}
	where the supremum is over all partitions $\mathcal{T}_{\alpha}=(t_{\alpha,i})_{i=0}^{n_\alpha}$ of $[a,b]$ satisfying $a=t_{\alpha,0} < t_{\alpha,1} < \cdots < t_{\alpha,n_{\alpha}}=b$.
\end{remark}
\begin{example}
	\label{ex:trop_circumference}
	The tropical circumference of a standard circle $C \subset \RR^2$ of radius $R$ equals the tropical perimeter of a circumscribing tropical polygon. It follows from the fact that by partitioning the circle into six arcs through the points of tangency ${\bf q_i}$, as seen in Figure~\ref{fig:circle_length}, we can already calculate the exact tropical length of the circle and any refined partitioning will not change the result.
	The length of each arc between two neighboring ${\bf q_i}$ is the absolute value of the difference in either their $x$ coordinates or their $y$ coordinates or the sum of both differences.
	The length of each of the arcs $({\bf q_1q_2}), ({\bf q_2q_3}), ({\bf q_4q_5})$ and $({\bf q_5q_0})$ is $\cos (\frac{\pi}{4})R = \frac{\sqrt{2}}{2}R$ (the difference in the $x$-coordinates or the $y$-coordinates) and the length of each of the arcs $({\bf q_0q_1})$ and $({\bf q_3q_4})$ is $2R$.
	Thus, the tropical circumference of the circle is
	\begin{equation}
		l_{\mathrm{tr}}(C) = (4 +2 \sqrt{2})R \approx 6.828R,
	\end{equation}
	which is greater than the circumference of a standard circle of the same radius.
\end{example}

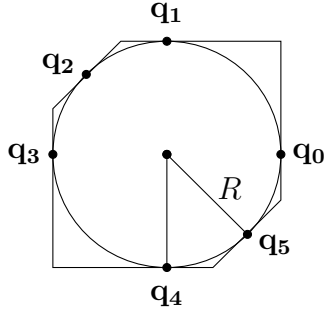
\begin{figure}[h]
	\centering
	\begin{tikzpicture}[scale=1.5]
		\draw (0,0);
		\begin{scope}
			\draw (-1.,-1.) -- (0.404,-1) -- (1,-0.404) -- (1,1) -- (-0.404,1) -- (-1,0.404) -- cycle;
			\filldraw [black] (0,0) circle (1.pt);
			\filldraw [black] (0,-1.) circle (1.pt);
			\filldraw [black] (0.707,-0.707) circle (1.pt);
			\filldraw [black] (1,0) circle (1.pt);
			\filldraw [black] (0,1) circle (1.pt);
			\filldraw [black] (-0.707,0.707) circle (1.pt);
			\filldraw [black] (-1,0) circle (1.pt);
			\draw[fill=none](0,0) circle (1.) node [black,yshift=-2cm] {};
			\draw[fill=black](0,0) circle (1 pt) node {};
			\draw(0,0) -- (0.707,-0.707) {};
			\draw(0,0) -- (0,-1) {};
			\node at (0.55,-0.3) {$R$};
			\node at (0,-1.25) {${\bf q_4}$};
			\node at (0.95,-0.8) {${\bf q_5}$};
			\node at (1.25,0) {${\bf q_0}$};
			\node at (0,1.25) {${\bf q_1}$};
			\node at (-0.95,0.8) {${\bf q_2}$};
			\node at (-1.25,0) {${\bf q_3}$};
		\end{scope}
	\end{tikzpicture}
	\caption{Tropical circumference of a standard circle}
	\label{fig:circle_length}
\end{figure}

\begin{definition}
	A curve $\gamma$ between the points $\varx$ and $\vary$ is a {\it tropical geodesic} if $l_{\mathrm{tr}}(\gamma) = d_{\mathrm{tr}}(\varx, \vary)$ (see also \cite[Ch. 5]{Jos21}).
\end{definition}

\begin{remarks}
	\begin{enumerate}
		\item Unless a standard line segment parallel to the tropical coordinate axes connects the points $\varx$ and $\vary$, there are infinitely many tropical geodesics between these points.
		\item Let $\gamma : [a,b] \to \Rn$ be a rectifiable curve with $\gamma(a)=\varx$ and $\gamma(b)=\vary$. Then $\gamma$ is a geodesic between $\varx$ and $\vary$ if and only if, for all $a \leq t < b$, the direction of $\gamma$ at time $t$ is confined to the parallelotope defined by $\gamma(t)$ and $\vary$ and which is bounded by hyperplanes of the form described in Theorem~\ref{thm:trop:geod_Rn}.
		In case $n=2$, a detailed description of the defined parallelogram is given in the proof of Proposition~\ref{prop:geod_two_points}. In fact, in $\Rn$, the tropical geodesic curves $\gamma$ between $\varx$ and $\vary$ are those for which, for each $i \neq j$, the orthogonal projection $p_{ij}(\gamma)$ on the $X_iX_j$-plane is a tropical curve in the plane between $p_{ij}(\varx)$ and $p_{ij}(\vary)$.
	\end{enumerate}
\end{remarks}
	
\begin{examples}
	\begin{enumerate}
		\item The standard line segment between $\varx$ and $\vary$ is a tropical geodesic.
		\item Each of the circle arcs between two consecutive points ${\bf q_i}$ in Example~\ref{ex:trop_circumference} is a tropical geodesic.
		\item We saw that the min tropical line segment from $\varx$ to $\vary$ in $\Rn$ is composed of specific standard line segments $\gamma_i$, $i=1,\ldots,m$, $m \leq n$. The max tropical line segment from $\varx$ to $\vary$ is composed of the same line segments but in reverse order.
		 Suppose now that we partition each line segment $\gamma_i$ into several subsegments $\gamma_{i_j}$ and then form a piecewise linear curve from $\varx$ to $\vary$ by concatenating all these line segments in any order we choose. Then the resulting piecewise linear curve is also a tropical geodesic. We say that such a curve is a {\it tropical piecewise linear geodesic} between $\varx$ and $\vary$.
	\end{enumerate}
	\end{examples}
Given two points $\varx, \vary \in \Rn$, we denote by $g_{\mathrm{tr}}(\{\varx, \vary \})$ the union of all tropical geodesics between $\varx$ and $\vary$.
The extension to $g_{\mathrm{tr}}(S)$ for an arbitrary set $S \subset \Rn$ is clear: $g_{\mathrm{tr}}(S) := \bigcup_{\varx, \vary \in S} \, g_{\mathrm{tr}}(\{\varx, \vary \})$.
\begin{proposition}
	The following are equivalent characterizations of $g_{\mathrm{tr}}(\{\varx, \vary \})$.
	\begin{enumerate}
		\item The union of all tropical piecewise linear geodesics between $\varx$ and $\vary$.
		\item The set of all points $\varz$, such that $d_{\mathrm{tr}}(\varx, \varz) + d_{\mathrm{tr}}(\varz, \vary) = d_{\mathrm{tr}}(\varx, \vary)$.
	\end{enumerate}
	
\end{proposition}
The proof is left to the reader.
The next definitions are similar to the definitions of standard convexity and convex hull in Euclidean geometry.
\begin{definition}
	A set $S \subset \Rn$ is {\it tropically geodesic} if it contains all tropical geodesics between every two points $\varx, \vary \in S$.
\end{definition}
\begin{remark}
	In \cite{Jos21} a tropically geodesic set is called bi-tropically convex.
\end{remark}
\begin{definition}
	Given a set $S \subset \Rn$, the {\it tropical geodesic hull} of $S$, $\mathrm{geod}_{\mathrm{tr}}(S)$, is the set satisfying each of the following equivalent statements.
	\begin{enumerate}
		\item The unique minimal tropically geodesic set that contains $S$.
		\item The intersection of all tropically geodesic sets that contain $S$.
		\item The set $\bigcup_{i \geq 0} g_{\mathrm{tr}}^i(S)$, where $g_{\mathrm{tr}}^0(S) = S$ and $g_{\mathrm{tr}}^{i+1}(S) = g_{\mathrm{tr}}(g_{\mathrm{tr}}^i(S))$, for $i \geq 0$. 
		\item The closure of $S$ under both min and max tropical linear combinations (see Subsection~\ref{subsec:tr_convexity} for the definitions). That is, the minimal set containing $S$ that is both min- and max-tropically convex.
	\end{enumerate} 
\end{definition}
\begin{example}
	Let $S = \{(0,0,0),(1,0,0),(1,1,0),(1,1,1)\} \subset \Rn$.
	These are the vertices of the three-dimensional simplex
	$$\Delta = \{ (x,y,z) \,:\, 0 \leq z \leq y \leq x \leq 1 \}$$ whose edges are all of tropical length 1.
	Then $g_{\mathrm{tr}}^i(S)$, for $i=0,1$, is the $i$-th skeleton of $\Delta$ and $g_{\mathrm{tr}}^2(S) = \Delta$. To see that $g_{\mathrm{tr}}^2(S) = \Delta$, let $\varp = (a,b,c)$ be a point of the simplex which is not on its one-dimensional skeleton. Then $\varp$ is on the tropical geodesic which is the line segment between the points $(\frac{a-b}{1-b},0,0)$ and $(1,1,\frac{c}{b})$ that are on the edges of $\Delta$. Clearly, $g_{\mathrm{tr}}^i(S)=g_{\mathrm{tr}}^2(S)$ for $i>2$, hence, $g_{\mathrm{tr}}^2(S) = \mathrm{geod}_{\mathrm{tr}}(S)$, the tropical geodesic hull of $S$.
\end{example}

\subsection{Compact tropically geodesic sets in the plane} 
\begin{proposition}
	\label{prop:geod_two_points}
	Let $\varp, \varq \in \RR^2$ be two points that are not on the same (standard) line in direction  $(1,0)$, $(0, 1)$ or $(1,1)$. Then $S=\mathrm{geod}_{\mathrm{tr}}(\{\varp, \varq\})$ is the region enclosed by the min and max tropical line segments between $\varp$ and $\varq$, that is,
	\begin{equation}
		\label{eq:parallelogram}
		S =\{(x,y) \in \RR^2 \,:\, (a \leq x \leq a'), (b \leq y \leq b'), (c \leq y-x \leq c')\},
	\end{equation}
	where $a,a',b,b',c,c' \in \RR$ define the supporting lines in directions $(0,1)$, $(1, 0)$ and $(1,1)$.
\end{proposition}
\begin{proof}
	Let $\varp = (p_1,p_2)$ and $\varq = (q_1,q_2)$.
	It suffices to consider all tropical {\it linear} geodesics between $\varp$ and  $\varq$. Up to interchanging between $\varp$ and $\varq$, we distinguish between three cases.
	
	Case 1: $p_1 < q_1$ and $p_2 > q_2$. Then the union of the min and max tropical line segments between the two points forms a rectangle. All tropical piecewise linear geodesics from $\varp$ to  $\varq$ consist of (standard) line segments in direction $(1,0)$, whose total length is $q_1-p_1$, and line segments in direction $(0,-1)$, whose total length is $p_2-q_2$. Clearly, all these piecewise linear curves are contained in the formed rectangle:
	$$
	p_1 \leq x \leq q_1, \qquad q_2 \leq y \leq p_2
	$$
	(where the bounds $q_2-q_1 \leq y-x \leq p_2-p_1$ are redundant)
	
	Case 2: $p_1 < q_1$ and $p_2 < q_2$, where $q_1-p_1 < q_2-p_2$. Then the union of the min and max tropical line segments between the two points forms a parallelogram. All tropical piecewise linear geodesics from $\varp$ to  $\varq$ consist of (standard) line segments in direction $(1,1)$, whose total tropical length is $q_1-p_1$, and line segments in direction $(0,1)$, whose total length is $(q_2-p_2) - (q_1-p_1)$. Then all these piecewise linear curves are contained in the formed parallelogram:
	$$
	p_1 \leq x \leq q_1, \qquad p_2-p_1 \leq y-x \leq q_2-q_1
	$$
	(where the bounds $p_2 \leq y \leq q_2$ are redundant).
	
	Case 3: $p_1 < q_1$ and $p_2 < q_2$, as in case 2, but with $q_1-p_1 > q_2-p_2$. 
	Then the union of the min and max tropical line segments between the two points forms a parallelogram. All tropical piecewise linear geodesics from $\varp$ to  $\varq$ consist of (standard) line segments in direction $(1,1)$, whose total tropical length is $q_2-p_2$, and line segments in direction $(1,0)$, whose total length is $(q_1-p_1) - (q_2-p_2)$. Then all these piecewise linear curves are contained in the formed parallelogram:
	$$
	p_2 \leq y \leq q_2, \qquad q_2-q_1 \leq y-x \leq p_2-p_1
	$$
	(where the bounds $p_1 \leq x \leq q_1$ are redundant).
\end{proof}
\begin{example}
	Figure~\ref{geod_two_points} shows the min and max tropical line segments, as well as the tropical geodesic hull, between two points in $\RR^2$.
\end{example}
\begin{figure}[t]
	\centering
	{\tiny
		\begin{tikzpicture}[scale=0.8]
			\filldraw [black] (0,0) circle (0.5mm);
			\filldraw [black] (1,2) circle (0.5mm);
			\draw (0,0) -- (1,1) -- (1,2);
			\node at (0.0,-.4) {\normalsize $\varp$};
			\node at (1.0,2.4) {\normalsize $\varq$};
			\node at (0.6,-1) {\normalsize (a)};
			\begin{scope}[xshift=0.3\textwidth]
				\filldraw [black] (0,0) circle (0.5mm);
				\filldraw [black] (1,2) circle (0.5mm);
				\draw (0,0) -- (0,1) -- (1,2);
				\node at (0.0,-.4) {\normalsize $\varp$};
				\node at (1.0,2.4) {\normalsize $\varq$};
				\node at (0.6,-1) {\normalsize (b)};
			\end{scope}
			\begin{scope}[xshift=0.6\textwidth]
				\filldraw [black] (0,0) circle (0.5mm);
				\filldraw [black] (1,2) circle (0.5mm);
				\draw (0,0) -- (1,1) -- (1,2) -- (0,1) -- cycle [fill=blue!40, opacity=0.5];
				\node at (0.0,-.4) {\normalsize $\varp$};
				\node at (1.0,2.4) {\normalsize $\varq$};
				\node at (0.6,-1) {\normalsize (c)};
			\end{scope}
		\end{tikzpicture}
	}
	\caption{(a) Min tropical line segment; (b) max tropical line segment; (c) tropical geodesic hull of $\{\varp,\varq\}$}
	\label{geod_two_points}
\end{figure}
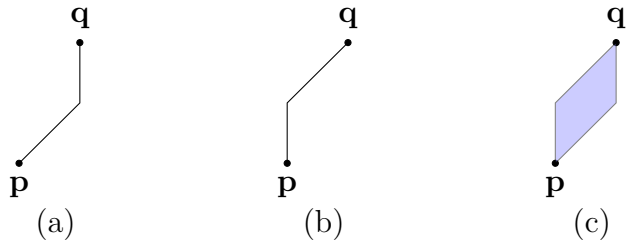
\begin{proposition}
	\label{prop:18forms}
	A non-empty compact set in $\RR^2$ that is tropically geodesic is either a point, or a line segment in direction $(1,0)$, $(0, 1)$ or $(1,1)$, or a two-dimensional tropical polygon (in fact, tropical triangle) in one of the 18 forms shown in Figure~\ref{fig:convexpolygons}.
\end{proposition}
\begin{proof}
 	Clearly, a point is tropically geodesic and the one-dimensional sets that are tropically geodesic are the line segments in direction $(1,0)$, $(0, 1)$ or $(1,1)$.
 	
	Otherwise, let	$S \subset \RR^2$ be a non-empty compact two-dimensional tropically geodesic set. Since, by Proposition~\ref{prop:geod_two_points}, $S$ is the union of parallelograms with edges in direction $(1,0)$, $(0, 1)$ or $(1,1)$, then the boundary of $S$ consists of line segments in these directions. Since $S$ is also (Euclidean) convex, then it must be the intersection of half-planes with boundaries in these directions. That is, $S$ is of the form 
	\begin{eqnarray}
		\label{eq:18forms}
		S &=&\{(x,y) \in \RR^2 \,:\, (a \leq x \leq a'), (b \leq y \leq b'), (c \leq y-x \leq c'),  \\
		&  &\mbox{for some } a,a',b,b',c,c' \in \RR \}. \nonumber
	\end{eqnarray}
	Therefore, the number of possible shapes  for $S$ is limited. When ignoring the lengths and proportions between the edges, there are only 18 such shapes (see Table~\ref{tab:polytypes}). Starting with the maximal number of edges, $6$, where there is only one such shape, the construction of the regions with $5,4$ and $3$ edges is achieved by all possibilities of deleting edges from the $6$-edges polygon, with the restriction that no adjacent edges can be deleted (see Figure~\ref{fig:convexpolygons}). Finally, it can be seen that each of these polygons consists of three min (max) tropical line segments, so they are tropical triangles, and no other tropical triangles exist.
	
	In the other direction, it is clear that each of the regions shown in Figure~\ref{fig:convexpolygons} is tropically geodesic.
 \end{proof}
\begin{table}[h]
	\centering
	\begin{tabular}{|c|c|} 
		\hline
		\# edges & \# types \\
		\hline \hline
		3 & 2  \\
		\hline
		4 & 9  \\
		\hline
		5 & 6  \\
		\hline
		6 & 1  \\
		\hline
	\end{tabular}	
	\vspace{5pt}
	\caption{Number of types of the different Euclidean convex tropical polygons}
	\label{tab:polytypes}
\end{table}
 \begin{figure}[t]
 	\centering
 	{\tiny
 		\begin{tikzpicture}[scale=0.8]
 			\draw (0,0) -- (1,0) -- (1,1) -- cycle [fill=blue!40, opacity=0.5];
 			\node at (0.6,-.25) {$1$};
 			\begin{scope}[xshift=0.16\textwidth]
 				\draw (0,0) -- (1,1) -- (0,1) -- cycle [fill=blue!40, opacity=0.5];
 				\node at (0.6,-.25) {$2$};
 			\end{scope}
 			\begin{scope}[xshift=.32\textwidth]
 				\draw (0,0) -- (2,0) -- (2,1) -- (0,1) -- cycle [fill=blue!40, opacity=0.5];
 				\node at (1.1,-.25) {$3$};
 			\end{scope}
 			\begin{scope}[xshift=.56\textwidth]
 				\draw (0,0) -- (2,0) -- (2,1) -- (1,1) -- cycle [fill=blue!40, opacity=0.5];
 				\node at (1.1,-.25) {$4$};
 			\end{scope}
 			\begin{scope}[xshift=.8\textwidth]
 				\draw (0,0) -- (1,0) -- (2,1) -- (0,1) -- cycle [fill=blue!40, opacity=0.5];
 				\node at (0.6,-.25) {$5$};
 			\end{scope}
 			
 			\begin{scope}[yshift=-.2\textwidth]
 				\draw (0,0) -- (1,1) -- (1,2) -- (0,1) -- cycle [fill=blue!40, opacity=0.5];
 				\node at (0.6,-.25) {$6$};
 			\end{scope}
 			\begin{scope}[xshift=.16\textwidth,yshift=-.2\textwidth]
 				\draw (0,0) -- (1,0) -- (1,2) -- (0,1) -- cycle [fill=blue!40, opacity=0.5];
 				\node at (0.6,-.25) {$7$};
 			\end{scope}
 			\begin{scope}[xshift=0.32\textwidth,yshift=-.2\textwidth]
 				\draw (0,0) -- (1,1) -- (1,2) -- (0,2) -- cycle [fill=blue!40, opacity=0.5];
 				\node at (0.6,-.25) {$8$};
 			\end{scope}
 			\begin{scope}[xshift=.43\textwidth,yshift=-.2\textwidth]
 				\draw (0,0) -- (1,0) -- (3,2) -- (2,2) -- cycle [fill=blue!40, opacity=0.5];
 				\node at (0.6,-.25) {$9$};
 			\end{scope}
 			\begin{scope}[xshift=.68 \textwidth,yshift=-.2\textwidth]
 				\draw (0,0) -- (2,2) -- (1,2) -- (0,1) -- cycle [fill=blue!40, opacity=0.5];
 				\node at (1.1,-.25) {$10$};
 			\end{scope}
 			\begin{scope}[xshift=.8\textwidth,yshift=-.2\textwidth]
 				\draw (0,0) -- (1,0) -- (2,1) -- (2,2) -- cycle [fill=blue!40, opacity=0.5];
 				\node at (0.6,-.25) {$11$};
 			\end{scope}
 			
 			\begin{scope}[yshift=-.4\textwidth]
 				\draw (0,0) -- (2,0) -- (2,2) -- (1,2) -- (0,1) -- cycle [fill=blue!40, opacity=0.5];
 				\node at (1.1,-.25) {$12$};
 			\end{scope}
 			\begin{scope}[xshift=.18\textwidth,yshift=-.4\textwidth]
 				\draw (0,0) -- (1,0) -- (2,1) -- (2,2) -- (0,2) -- cycle [fill=blue!40, opacity=0.5];
 				\node at (0.6,-.25) {$13$};
 			\end{scope}
 			\begin{scope}[xshift=0.35\textwidth,yshift=-.4\textwidth]
 				\draw (0,0) -- (1,0) -- (3,2) -- (1,2) -- (0,1) -- cycle [fill=blue!40, opacity=0.5];
 				\node at (0.6,-.25) {$14$};
 			\end{scope}
 			\begin{scope}[xshift=.47\textwidth,yshift=-.4\textwidth]
 				\draw (0,0) -- (2,0) -- (3,1) -- (3,2) -- (2,2) -- cycle [fill=blue!40, opacity=0.5];
 				\node at (1.1,-.25) {$15$};
 			\end{scope}
 			\begin{scope}[xshift=.69 \textwidth,yshift=-.4\textwidth]
 				\draw (0,0) -- (1,0) -- (1.5,0.5) -- (1.5,2) -- (0,0.5) -- cycle [fill=blue!40, opacity=0.5];
 				\node at (0.6,-.25) {$16$};
 			\end{scope}
 			\begin{scope}[xshift=.83\textwidth,yshift=-.4\textwidth]
 				\draw (0,0) -- (1.5,1.5) -- (1.5,2) -- (0.5,2) -- (0,1.5) -- cycle [fill=blue!40, opacity=0.5];
 				\node at (.85,-.25) {$17$};
 			\end{scope}
 			\begin{scope}[yshift=-.6\textwidth]
 				\draw (0,0) -- (2,0) -- (3,1) -- (3,2) -- (1.3,2) -- (0,0.7) -- cycle [fill=blue!40, opacity=0.5];
 				\node at (1.1,-.25) {$18$};
 			\end{scope}
 		\end{tikzpicture}
 	}
 	\caption{Compact tropically geodesic 2D sets in $\RR^2$}
 	\label{fig:convexpolygons}
 \end{figure}
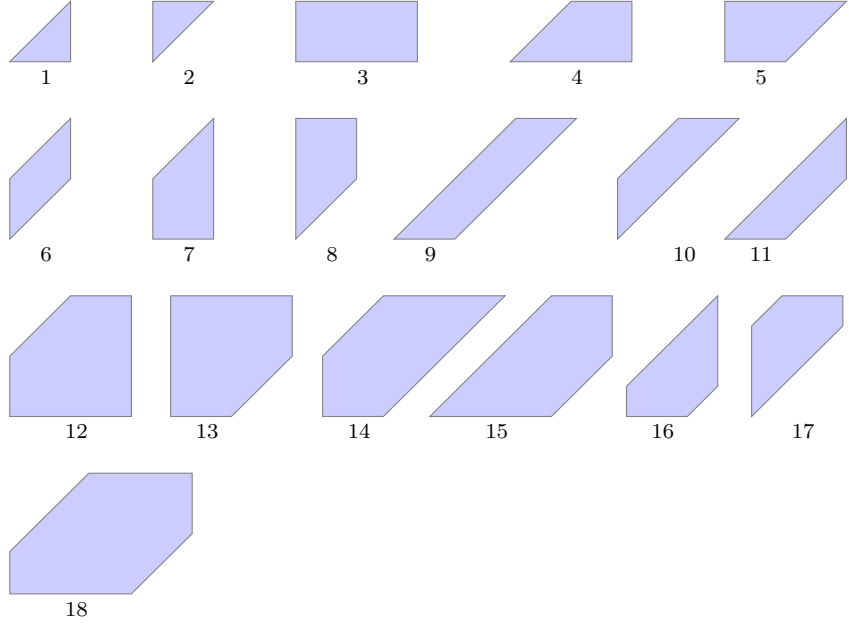

 \subsection{Compact tropically geodesic sets in $\Rn$}
 Compact tropically geodesic sets in $\Rn$ are characterized in a similar way to those in the plane, forming convex polytopes, and there are only finitely many types of them. 
\begin{theorem}
	\label{thm:trop:geod_Rn}
	Let $S$ be a non-empty compact tropically geodesic set in $\Rn$. Then $S$ is a region which is the intersection of $n(n+1)$ half-spaces of the form:
	\begin{enumerate}[label=(\roman*)]
		\item $x_i \geq a_i$ and $x_i \leq a'_i$, $a_i \leq a'_i$, for $i=1,\ldots,n$;
		\item $x_i-x_j \geq b_{ij}$, $b_{ij} \leq -b_{ji}$, for $i = 1,\ldots,n$, $j \neq i$,
	\end{enumerate}
	where the $a_i,a'_i,b_{ij} \in \RR$ define the supporting hyperplanes.
	And vice versa: every non-empty compact set $S \subset \Rn$ that satisfies the above conditions is tropically geodesic.
\end{theorem}
\begin{proof}
	First, note that (ii) is equivalent to the non-emptiness of the intersection of the half-spaces $x_i-x_j \geq b_{ij}$ and $x_j-x_i \geq b_{ji}$, for all $i \neq j$.
	
	Let $\varx, \vary$ be two points in $S$ and let $\gamma$ be a tropical piecewise linear geodesic between $\varx$ and $\vary$. Then the orthogonal projection $p_{ij}(\gamma)$ to the plane $X_iX_j$, $i \neq j$, is a tropical piecewise linear geodesic between $p_{ij}(\varx)$ and $p_{ij}(\vary)$. And vice versa: every tropical piecewise linear geodesic between $p_{ij}(\varx)$ and $p_{ij}(\vary)$ lifts to a tropical piecewise linear geodesic between $\varx$ and $\vary$.
	
	It follows, by Proposition~\ref{prop:18forms}, that $p_{ij}(S)$, is a tropically geodesic region of the form of \eqref{eq:18forms}.
	Since $S$ is a (Euclidean) convex set then it has the form of intersection of half-spaces as described in the statement of the proposition.
	
	In the other direction, let $S \subset \Rn$ be a non-empty compact set that is the intersection of half-spaces of the forms described in (i) and (ii). Then these conditions are exactly the ones that guarantee that any tropical piecewise linear geodesic between two points $\varx, \vary \in S$ stays in $S$, implying that $S$ is tropically geodesic by definition. 
\end{proof} 
\begin{remark}
		If $S$ is tropically geodesic but not compact then a more general version of Theorem~\ref{thm:trop:geod_Rn} holds: $S$ may be bounded by only some of the half-spaces (e.g., when $x_i$ is not bounded from above) or some inequalities that involve parts of the bounding hyperplanes may be strict inequalities.  
\end{remark}
 
\section{The tropical unit ball}
\label{sec:ball}
In this section we describe the tropical (closed) unit ball in $\Rn$ and its basic properties (see also \cite{CJS22}). 
\begin{definition}
	The $n$-dimensional {\it tropical (closed) ball} of radius $R$ with center at $\varp \in \Rn$ is
	\begin{equation}
		\BRn(\varp) := \{\varx \in {\Rn} : d_{\mathrm{tr}}(\varp, \varx) \leq R \}.
	\end{equation}
\end{definition}
We denote by $\BRn$ the tropical ball of radius $R$ that is centered at the origin and when in addition the radius is 1 then it is denoted by $\Bn$.

The tropical ball is a polytope which is a standard convex set as well as a tropically convex and a tropically geodesic set.
When centered at the origin, it is centrally symmetric but not a regular polytope: it is stretched along the directions $\bf 1$ and $-\bf 1$ and shrunk in the orthogonal directions of the hyperplane $x_1 + \cdots + x_n = 0$
(see Figure~\ref{fig:tr_balls}). The two points on the boundary of $\Bn$ that are furthest from the center are $\bf 1$ and $-\bf 1$, which are at (Euclidean) distance $\sqrt{n}$ from the center (so that, the diameter of $\Bn$ is $2\sqrt{n}$), tending to $\infty$ as $n \to \infty$, whereas the closest points are those with one coordinate $\half$, one coordinate $-\half$ and the other coordinates $0$, being at distance $\frac{\sqrt{2}}{2}$ from the center.
\begin{observation}
	\label{obs:ball_conditions}
	$\Bn$ consists of the set of points $\varx = (x_1, \ldots, x_n)$ that satisfy these two conditions:
	\begin{enumerate}[label=(\roman*)]
		\item $\displaystyle{\max_i \{ \lvert x_i \rvert \} \leq 1}$;
		\item $\displaystyle{\max_i \{x_i\} - \min_j \{x_j\} \leq 1}$.
	\end{enumerate}
	Equivalently, fixing a dummy variable $x_0=0$, then
	\begin{equation}
		\label{eq:tr_unit_ball}
		\Bn = \{(x_1, \ldots,  x_n) : \vert x_i-x_j \vert \leq 1, \, 0 \leq i,j \leq n \} \}.
	\end{equation}
\end{observation}
\noindent
{\it Vertices:} The $2^{n+1}-2$ vertices of $\Bn$ are:
\begin{enumerate}[label=(\roman*)]
	\item $\{(x_1, \ldots, x_n) \in \{0,1\}^n \} \setminus\{\Zero\}$;
	\item $\{(x_1, \ldots, x_n) \in \{0,-1\}^n \} \setminus\{\Zero\}$.
\end{enumerate}
Here $\Zero = (0, \ldots, 0)$. \\
{\it Hyperplanes:} $\Bn$ is supported by the following $n(n+1)$ hyperplanes (with a fixed $x_{0} = 0$ ): \\
\indent\hspace{2mm} $x_i-x_j = 1$, for $i,j = 0,\ldots,n, i \neq j$. \\
{\it Facets:} The $n(n+1)$ facets (of dimension $n-1$) of $\Bn$ are:
\begin{enumerate}[label=(\roman*)]
	\item $F_i = \{(x_1, \ldots, x_n) : x_i =  1, 0 \leq x_j \leq 1, j \neq i\}$, for $i=1,\ldots,n$;
	\item $F_{-i}=\{(x_1, \ldots, x_n) : x_i =  -1, -1 \leq x_j \leq 0, j \neq i\}$, for $i=1,\ldots,n$;
	\item $F_{ij}=\{(x_1, \ldots, x_n) : 0 \leq x_i \leq 1, x_j = x_i-1, x_j \leq x_k \leq x_i,  k=1,\ldots,n, k \neq i,j\}$, for $i,j=1,\ldots,n, i \neq j$ (when $n>1$).
\end{enumerate}
We have, $F_{-i}=-F_i$ and $F_{ji}=-F_{ij}$.
The union of the above facets is the {\it tropical unit sphere} $\San$.

\begin{figure}[h]
	\centering
	\begin{tikzpicture}
		\begin{axis}[
			axis x line=middle,
			axis y line=middle,
			grid = major,
			width=7cm,
			height=7 cm,
			grid style={dashed, gray!80},
			xmin=-2.0,     
			xmax= 2.0,    
			ymin= -2.0,     
			ymax= 2.0,   
			xlabel=$x$,
			ylabel=$y$,
			/pgfplots/xtick={-1.0, 0.0, 1.0}, 
			/pgfplots/ytick={-1.0, 0.0, 1.0}, 
			]
			\draw[thick,black] (1.0,0.0) -- (1.0,1.0) -- (0.0,1.0) -- (-1.0,0.0) -- (-1.0,-1.0) -- (0.0,-1.0) -- cycle [fill=blue!20, opacity=0.3];
			\draw[thick,black] (1.0,0.0) -- (1.0,1.0) -- (0.0,1.0) -- (-1.0,0.0) -- (-1.0,-1.0) -- (0.0,-1.0) -- cycle;
		\end{axis}
	\end{tikzpicture}
	\qquad
	\begin{tikzpicture}
		\begin{axis}[
			view={50}{10},
			axis x line=middle,
			axis y line=middle,
			axis z line=middle,
			grid = major,
			width=9cm,
			height=9cm,
			grid style={dashed, gray !80},
			xmin=-2,     
			xmax= 2,    
			ymin= -2,     
			ymax= 2,   
			zmin= -2,     
			zmax= 2,   
			xlabel=$x$,
			ylabel=$y$,
			zlabel=$z$,
			/pgfplots/xtick={-1.0, 0.0, 1.0}, 
			/pgfplots/ytick={-1.0, 0.0, 1.0}, 
			/pgfplots/ztick={-1.0, 0.0, 1.0}, 
			]
			\coordinate (A1) at (1,0,0);
			\coordinate (A2) at (1,1,0);
			\coordinate (A3) at (0,1,0);
			\coordinate (A4) at (0,0,1);
			\coordinate (A5) at (1,0,1);
			\coordinate (A6) at (1,1,1);
			\coordinate (A7) at (0,1,1);

			\coordinate (B1) at (-1,0,0);
			\coordinate (B2) at (-1,-1,0);
			\coordinate (B3) at (0,-1,0);
			\coordinate (B4) at (0,0,-1);
			\coordinate (B5) at (-1,0,-1);
			\coordinate (B6) at (-1,-1,-1);
			\coordinate (B7) at (0,-1,-1);

			\draw [thick] (A1) -- (A2) -- (A6) -- (A5) -- cycle;
			\draw (A3) -- (A2) -- (A6) -- (A7) -- cycle;
			\draw [thick] (A4) -- (A5) -- (A6) -- (A7) -- cycle;
			\draw (B1) -- (B2) -- (B6) -- (B5) -- cycle;
			\draw [thick] (B3) -- (B2) -- (B6) -- (B7) -- cycle;
			\draw (B4) -- (B5) -- (B6) -- (B7) -- cycle;
			\draw (A7) -- (B1);
			\draw [thick] (A4) -- (B2);
			\draw (A3) -- (B5);
			\draw [thick] (A5) -- (B3);
			\draw (A2) -- (B4);
			\draw [thick] (A1) -- (B7);
			\draw [thick] (B7) -- (B4) -- (A2);
			\draw [thick] (A2) -- (B4) -- (B7) -- (B6) -- (B2) -- (A4) -- (A7) -- (A6) -- cycle [fill=blue!20, opacity=0.3];
		\end{axis}
	\end{tikzpicture}
	\caption{The tropical unit ball in the plane (a hexagon, left) and in space (a rhombic dodecahedron with four-sided faces, right)}
	\label{fig:tr_balls}
\end{figure}

In a standard $n$-dimensional ball of radius $R$, two points $\varp$ and $\varq$ are of distance $2R$ from each other if and only if they are antipodal points on the surface of the ball, so that the line segment joining them passes through the center of the ball.
In a tropical $n$-dimensional ball $\BRn$ of radius $R$, two points $\varp$ and $\varq$ are of tropical distance $2R$ from each other if and only if there is a tropical geodesic joining them that passes through the center of $\BRn$, and this happens if and only if $\varp$ and $\varq$ belong to opposite facets of $\BRn$: $\varp \in F_{i}$ and $\varq \in F_{-i}$ or $\varp \in F_{ij}$ and $\varq \in F_{ji}$. It means that, when $n>1$, given a point $\varp \in \partial \BRn$, there are infinitely many points $\varq \in \partial \BRn$, such that $d_{\mathrm{tr}}(\varp, \varq) = 2R$, the tropical diameter of $\BRn$.

\subsection{Tropical coordinate system}
\label{subsec:coord}
Let the {\it (min) tropical standard unit vectors} be ${\bf \tilde{e}_1} = (1,0,\ldots,0), \ldots, {\bf \tilde{e}_n} = (0,\ldots,0,1)$ and ${\bf \tilde{e}_{n+1}} = -\One =(-1,\ldots,-1)$.
Then we define the {\it (min) tropical coordinate system} of $\Rn$ with respect to the axes in the $n+1$ directions of these unit vectors.
$\Rn$ is then decomposed into $n+1$ orthants $\Rn_j$, for $j=1,\ldots, n+1$, where $\Rn_j$ is the positive cone determined by the $n$ vectors ${\bf \tilde{e}_1},\ldots,{\bf \widehat{\tilde{e}}_j},\ldots,{\bf \tilde{e}_{n+1}}$.
\begin{remark}
		In the max setting the coordinate axes and the positive cones are at directions opposite to the ones in the min setting.
\end{remark}
In Section~\ref{sec:proj} we (arbitrarily) chose to map the element $(x_1 : x_2 : \ldots : x_{n+1}) \in \Projn$ to the element $(x_1-x_{n+1}, x_2-x_{n+1}, \ldots, x_n-x_{n+1}) \in \Rn$. When expressing this element with respect to the tropical coordinate system (written in square brackets) we have:
\begin{eqnarray*}
	&&(x_1-x_{n+1}, x_2-x_{n+1}, \ldots, x_n-x_{n+1}) = \\
	&&[x_1-x_{n+1}, x_2-x_{n+1}, \ldots, x_n-x_{n+1},0] = \\
	&&[x_1, x_2, \ldots, x_{n+1}],
\end{eqnarray*}
where the last equality follows from the non-uniqueness of this representation:
$[a_1, \ldots, a_{n+1}] = [a_1+a, \ldots, a_{n+1}+a]$ for all $a \in \RR$.
It follows that we can choose a representative of each such equivalence class in which (at least) one coordinate is zero and the other coordinates are non-negative, and then we have a unique representation.
In this case, when  the $j$-th entry is zero, we use the following notation:
\begin{equation}
	\label{eq:j-coord}
[x_1, \ldots, x_{j-1}, 0, x_{j+1}, \ldots  x_{n+1}] = (x_1, \ldots, x_{j-1}, x_{j+1}, \ldots  x_{n+1})_j,
\end{equation}
where all the $x_i$ are non-negative. If more than one entry is zero then we can choose which one to omit.

\subsection{Minkowski sum}
One way of obtaining $\Bn$ is by taking the positive unit hypercube $I^n$, defined by the $n$ standard unit vectors ${\bf \tilde{e}_1}, \ldots, {\bf \tilde{e}_n}$, and moving it in direction $-\One$ until it coincides with the negative unit hypercube, defined by $-{\bf \tilde{e}_1}, \ldots, -{\bf \tilde{e}_n}$. 
That is, $\Bn$ is the parallelotope defined by ${\bf \tilde{e}_1}, \ldots, {\bf \tilde{e}_{n+1}}$.
This can be written in terms of Minkowski sum.
\begin{definition}
	The {\it Minkowski sum} of two sets $A$ and $B$ of position vectors is
	$U+V := \{{\bf u}+{\bf v} : {\bf u} \in U, {\bf v} \in V\}$.
\end{definition}
Let $I_{{\bf \tilde{e}_1}}, \ldots, I_{{\bf \tilde{e}_{n+1}}}$ be the tropical unit segments in  directions ${\bf \tilde{e}_1}, \ldots, {\bf \tilde{e}_{n+1}}$.
We show that $\Bn$ is the Minkowski sum of these segments, that is, $\Bn$ is a {\it zonotope}.
\begin{proposition}
	\label{prop:Minkowski}
	$\Bn = I_{{\bf \tilde{e}_1}} + \cdots + I_{{\bf \tilde{e}_{n+1}}}$.
\begin{proof}
	In fact, the validity of this statement follows from the above description of the facets of the tropical unit ball.
	Let us see it directly from the definition of $\Bn$. Clearly, the set of points of $\Bn$ with non-negative entries is exactly the positive unit hypercube $I^n$. We can define $I^n$ by
	$$
	I^n := \{ (x_1,\ldots,x_n) \in \Rn : 0 \leq x_i \leq 1 \; \wedge \; \vert x_i - x_j \vert \leq 1, \text{ for all } i,j \},
	$$
	where the second set of inequalities is implied by the first ones.
	It is known, and easy to see, that $I^n = I_{{\bf \tilde{e}}_1} + \cdots + I_{{\bf \tilde{e}}_{n}}$.
	But then
	$$
	I^n + I_{{\bf \tilde{e}_{n+1}}} = \{ (x_1,\ldots,x_n) \in \Rn : \vert x_i \vert \leq 1 \; \wedge \; \vert x_i - x_j \vert \leq 1, \text{ for all } i,j \},
	$$
	showing that $I_{{\bf \tilde{e}_1}} + \cdots + I_{{\bf \tilde{e}_{n+1}}} = \Bn$.
%
	%
%
\end{proof}
\end{proposition}
By Proposition~\ref{prop:Minkowski}, we have another way to describe $\Bn$:
\begin{equation}
	\Bn = \{ \varx = \sum_{i=1}^{n+1} a_i {\bf \tilde{e}_i} : 0 \leq a_i  \leq 1, \text{ for all } i \}.
\end{equation}

\subsection{The tropical angle between two (standard) lines}
\begin{definition}
	Let $\San(\varp)$ be the tropical unit sphere with center at $\varp$ and let $\varx,\vary \in \San(\varp)$. Then the intrinsic distance between $\varx$ and $\vary$ on $\San(\varp)$ is
	$$
		d^{\circ}_{\mathrm{tr}}(\varx, \vary) := \inf \{ l_{\mathrm{tr}}(\gamma) : \gamma \subset \San(\varp), \gamma(0)=\varx, \gamma(1)=\vary \}.
	$$
\end{definition}

\begin{definition}
	Let $l_1,l_2 \subset \Rn$, $n \geq 2$, be two (standard) lines that meet at a common point $\varp$. Let $\San(\varp)$ be the tropical unit sphere with center at $\varp$ and let $\bf q_{11}, q_{12}$ and $\bf q_{21}, q_{22}$ be the points where $l_1$ and $l_2$ intersect $\San(\varp)$. Then we define the {\it tropical angle} (in radians) between $l_1$ and $l_2$ to be
	$$
		\angle_{\mathrm{tr}}(l_1,l_2) := \min_{j,k} d^{\circ}_{\mathrm{tr}}({\bf q_{1j}, q_{2k}}).
	$$
\end{definition}
\begin{remark}
	The above definition of a tropical angle is different from the definition given in \cite{ABGJ14}, in which the tropical angle is either $\frac{\pi}{2}$ or $0$.
\end{remark}

The circumference of the tropical unit circle is 6 (tropical) radians and, as expected, we have the following result.
\begin{proposition}
	Let $l_1,l_2 \subset \Rn$ be two intersecting (standard) lines. Then
	$$
	0 \leq \angle_{\mathrm{tr}}(l_1,l_2) \leq 3.
	$$
\end{proposition}
\begin{proof}
	By the above definition of a tropical angle, we need to show that, for every $\varx,\vary \in \San$, we have $0 \leq d^{\circ}_{\mathrm{tr}}(\varx,\vary) \leq 3$, where the left inequality is clear. First, we show that when $\varx = (x_1,\ldots,x_n) \in \San$ then $d^{\circ}_{\mathrm{tr}}(\varx,\One) + d^{\circ}_{\mathrm{tr}}(\varx,-\One) =3$.
	
	Case (i): $\varx$ is on a facet of the form $x_i=1$, for some $1 \leq i \leq n$
	(similarly, for a facet $x_i=-1$). Assume, w.l.o.g., that $i=n$ and $0 \leq x_1 \leq x_2 \leq \cdots \leq x_{n-1} \leq x_n = 1$. We have $d^{\circ}_{\mathrm{tr}}(\varx,\One) = d_{\mathrm{tr}}(\varx,\One) = 1- x_1$ via a tropical geodesic on the facet $x_n=1$. Then a shortest path from $\varx$ to $-\One$ starts on the facet $x_n=1$ in direction $(-x_1,-x_1,\ldots,-x_1,0)$ and reaches the point $(0,x_2-x_1,\ldots,x_{n-1}-x_1,1) \in (x_n=1) \cap (x_n-x_1=1)$ on a path of tropical length $x_1$. The next line segment towards $-\One$ of shortest tropical length is in direction $-\One$ on the facet $(x_n-x_1=1)$, reaching the point $\vary = (-1,x_2-x_1-1,\ldots,x_{n-1}-x_1-1,0) \in (x_n-x_1=1) \cap (x_1=-1)$ and is of tropical length $1$. Finally, a tropical geodesic on the facet $x_1=-1$ reaches $-\One$ with $d^{\circ}_{\mathrm{tr}}(\vary,-\One) = d_{\mathrm{tr}}(\vary,-\One) = 1$.
	We have $d^{\circ}_{\mathrm{tr}}(\varx,-\One) =x_1+1+1=2+x_1$ and $d^{\circ}_{\mathrm{tr}}(\varx,\One) + d^{\circ}_{\mathrm{tr}}(\varx,-\One) =(1-x_1)+(2+x_1)=3$.
	
	Case (ii): $\varx$ lies on a facet of the form $x_i-x_j=1$, $0<x_i<1$, for some $i \neq j$. Assume, w.l.o.g., that $0 > x_1 \leq x_2 \leq \cdots \leq x_{n-1} \leq x_n =  x_1+1 >0$. Then a path towards $\One$ that starts at the facet of $\varx$ in direction $\One$ is of shortest tropical length $-x_1$, reaching the point $\vary = (0,x_2-x_1-1,\ldots,x_{n-1}-x_1,1)$. Then $d^{\circ}_{\mathrm{tr}}(\vary,\One) = d_{\mathrm{tr}}(\vary,\One) = 1$ and  $d^{\circ}_{\mathrm{tr}}(\varx,\One) = 1-x_1$. A similar computation gives $d^{\circ}_{\mathrm{tr}} (\varx,-\One) = 2+x_1$ and overall, $d^{\circ}_{\mathrm{tr}}(\varx,\One) + d^{\circ}_{\mathrm{tr}}(\varx,-\One) =(1-x_1)+(2+x_1)=3$.
	
	From the above, it follows that $d^{\circ}_{\mathrm{tr}}(\One, -\One) = 3-0=3$, and for arbitrary $\varx,\vary \in \San$, $(d^{\circ}_{\mathrm{tr}}(\varx,\One) + d^{\circ}_{\mathrm{tr}}(\One,\vary)) + (d^{\circ}_{\mathrm{tr}}(\varx,-\One) + d^{\circ}_{\mathrm{tr}}(-\One,\vary)) = 6$, implying that $0 \leq d^{\circ}_{\mathrm{tr}}(\varx,\vary) \leq 3$.
\end{proof}
	We say that the $n+1$ tropical unit vectors ${\bf \tilde{e}_1}, \ldots, {\bf \tilde{e}_{n+1}}$ in the directions of the tropical coordinate axes in $\Rn$, $n \geq 2$, are pairwise {\it tropically orthogonal}.
	Note that the tropical angle between any two of these vectors is 2 radians, which is a third of a full tropical circle, in contrast to standard orthogonal vectors, which form an angle of $\frac{\pi}{2}$ radians, a quarter of a full circle.
	Of course, this is not the only possible definition of tropical orthogonality. A natural choice is having a tropical scalar product $\infty$, the neutral element for tropical addition, but this definition applies to $\TT = \RR \cup \infty$ and not to $\RR$.

\subsection{Decomposition into tropical unit hypercubes}
\begin{definition}
	The $n$-dimensional (min) {\it tropical unit hypercube} of type $j$, for $j=1,\ldots,n+1$, with {\it base point} at the origin, denoted $\In{j}$, is the zonotope defined by the $n$ tropically orthogonal tropical unit vectors
	${\bf \tilde{e}_1},\ldots,{\bf \widehat{\tilde{e}}_j},\ldots,{\bf \tilde{e}_{n+1}}$.
	When the base point is at the point $\varp$ (a translation of $\In{j}$ by $\varp$)
	then it is denoted $\In{j}(\varp)$. 
\end{definition}
\begin{remarks}
	\begin{enumerate}
		\item The zonotopes $\In{j}$, $j \neq n+1$, are defined by $n$ tropically orthogonal line segments of the same tropical length and that is why we use the term hypercube. If, however, we look at the tropical angles between the edges then they are not all equal and that is because we define lines that are parallel to the tropical coordinate axes to be tropically orthogonal, forming a tropical angle of 2 radians and not of $\frac{6}{4}=1.5$ radians.
		\item The tropical unit hypercubes with base point at the origin in the max-plus setting are defined by $-{\bf \tilde{e}_1},\ldots,-{\bf \widehat{\tilde{e}}_j},\ldots,-{\bf \tilde{e}_{n+1}}$.
	\end{enumerate}
\end{remarks}
\begin{proposition}
	\label{prop:compose}
	The tropical unit ball $\Bn$ is the union of the $n+1$ tropical unit hypercubes $\In{j}$, $j = 1, \ldots, n+1$, where the union is disjoint up to common facets.
\end{proposition}
\begin{proof}
	For each $j$, $j = 1, \ldots, n+1$,$\In{j}$ is exactly the part of $\Bn$ in the orthant defined by the coordinate axes in directions ${\bf \tilde{e}_1},\ldots,{\bf \widehat{\tilde{e}}_j},\ldots,{\bf \tilde{e}_{n+1}}$.
\end{proof}
In Figure~\ref{fig:decomposition} we see the decomposition into tropical hypercubes in the plane.
	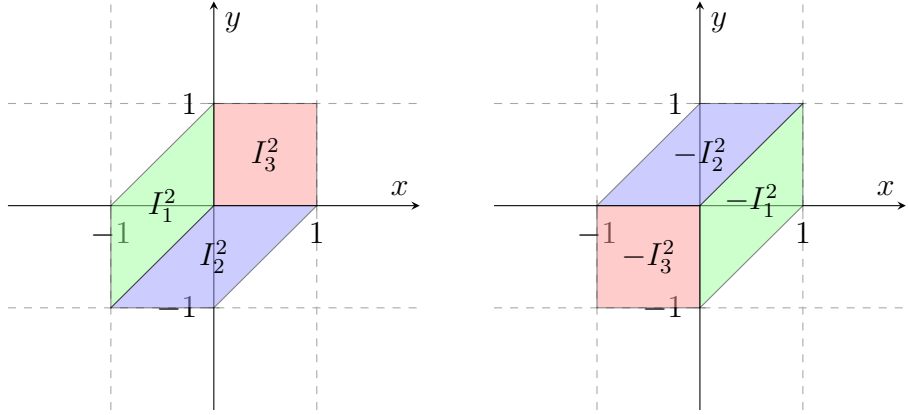
\begin{figure}[h]
		\centering
		\begin{tikzpicture}
		\begin{axis}[
			axis x line=middle,
			axis y line=middle,
			grid = major,
			width=7cm,
			height=7 cm,
			grid style={dashed, gray!70},
			xmin=-2.0,     
			xmax= 2.0,    
			ymin= -2.0,     
			ymax= 2.0,   
			xlabel=$x$,
			ylabel=$y$,
			/pgfplots/xtick={-1.0, 0.0, 1.0}, 
			/pgfplots/ytick={-1.0, 0.0, 1.0}, 
			]
			\draw (0.0,0.0) -- (1.0,0.0) -- (1.0,1.0) -- (0.0,1.0) -- cycle [fill=red!40, 	opacity=0.5];
			\draw (0.0,0.0) -- (0.0,1.0) -- (-1.0,0.0) -- (-1.0,-1.0) -- cycle [fill=green!40, 	opacity=0.5];
			\draw (0.0,0.0) -- (-1.0,-1.0) -- (0.0,-1.0) -- (1.0,0.0) -- cycle [fill=blue!40, 	opacity=0.5];
			\node at (0.5,0.5) {$I^2_3$};
			\node at (-0.5,0.0) {$I^2_1$};
			\node at (0.0,-0.5) {$I^2_2$};
		\end{axis}
	\end{tikzpicture}
	\qquad
	\begin{tikzpicture}
		\begin{axis}[
			axis x line=middle,
			axis y line=middle,
			grid = major,
			width=7cm,
			height=7 cm,
			grid style={dashed, gray!70},
			xmin=-2.0,     
			xmax= 2.0,    
			ymin= -2.0,     
			ymax= 2.0,   
			xlabel=$x$,
			ylabel=$y$,
			/pgfplots/xtick={-1.0, 0.0, 1.0}, 
			/pgfplots/ytick={-1.0, 0.0, 1.0}, 
			]
			\draw (0.0,0.0) -- (-1.0,0.0) -- (-1.0,-1.0) -- (0.0,-1.0) -- cycle [fill=red!40, opacity=0.5];
			\draw (0.0,0.0) -- (0.0,-1.0) -- (1.0,0.0) -- (1.0,1.0) -- cycle [fill=green!40, opacity=0.5];
			\draw (0.0,0.0) -- (1.0,1.0) -- (0.0,1.0) -- (-1.0,0.0) -- cycle [fill=blue!40, opacity=0.5];
			\node at (-0.5,-0.5) {$-I^2_3$};
			\node at (0.5,0.05) {$-I^2_1$};
			\node at (0.0,0.5) {$-I^2_2$};
		\end{axis}
	\end{tikzpicture}
	\caption{Decomposition of the 2-dimensional tropical unit ball into tropical unit hypercubes: min-plus decomposition (left) and max-plus decomposition (right)}
	\label{fig:decomposition}
\end{figure}

\subsection{Tropical geodesic hull}
We show here that the tropical unit ball is the tropical geodesic hull of the points defined by the tropical unit vectors $\bf \tilde{e}_i$.
\begin{proposition}
	\label{prop:geod_hull}
	$\Bn=\mathrm{geod}_{\mathrm{tr}}(\{{\bf \tilde{e}_1}, \ldots, {\bf \tilde{e}_{n+1}}\})=\mathrm{geod}_{\mathrm{tr}}(\{-{\bf \tilde{e}_1}, \ldots, -{\bf \tilde{e}_{n+1}}\})$.
\end{proposition}
\begin{proof}
	Given the set $S=\{{\bf \tilde{e}_1}, \ldots, {\bf \tilde{e}_{n+1}}\}$ then by  Theorem~\ref{thm:trop:geod_Rn}, $\mathrm{geod}_{\mathrm{tr}}(S)$ is the intersection of the half-spaces $-1 \leq x_i$, $x_i \leq 1$, for $i=1,\ldots,n$, and $-1 \leq x_i-x_j$, $x_i-x_j \leq 1$, for $i = 1,\ldots,n$, $j \neq i$, which coincides with the description of the supporting hyperplanes of $\Bn$ given in Observation~\ref{obs:ball_conditions}. The same holds for $S=\{-{\bf \tilde{e}_1}, \ldots, -{\bf \tilde{e}_{n+1}}\}$.
\end{proof}

\subsection{Tropical linear combination and tropical convexity}
\label{subsec:tr_convexity}
The following definitions are with regard to min-plus operations, similar definitions hold in the max-plus setting.
The notion of tropical convexity was introduced by Develin and Sturmfels \cite{DS04} and by Cohen, Gaubert, Quadrat and Singer \cite{CGQS05} (see also \cite{MS15}, p. 228). 
A set $S \subseteq \Ran$ is {\it tropically convex} if it is closed under tropical linear combinations: for all $\varx, \vary \in S$ and $a,b \in \RR$, we have $a \tm \varx \ta b \tm \vary \in S$. 
The {\it tropical convex hull} of a set $S$ is the set of all tropical linear combinations of the elements of $S$.

Let us now see what happens when a tropical linear combination in $\Projn$ is mapped to $\Rn$. 
As before, we represent each equivalent class in $\Projn$ by the unique element with last coordinate zero.
Then a min linear combinations $a_1 \tm {\bf x_1} \ta \cdots \ta a_r \tm {\bf x_r}$ of such representatives can always be chosen, by adding the same scalar $c$ to each coefficient, so that one of the $a_j$ is zero and the other coefficients are non-negative. Then the linear combination is similar to the standard definition of convexity since 0 is the identity element of tropical multiplication (the analogue of standard $1$) and $\infty$ (when working over $\TT$ and not over $\TT^{\times}$) is the identity element of tropical addition (the analogue of standard $0$), so that we have all coefficients of the linear combination between the tropical zero element and the tropical one element and their sum (the tropical minimum) is the tropical one element ($0$ in the tropical setting). This means that the set of all tropical linear combinations $a_1 \tm {\bf x_1} \ta \cdots \ta a_r {\bf x_r}$
is the same as the set of all convex combinations $a_1 \tm {\bf x_1} \ta \cdots \ta a_r {\bf x_r}$, where the $a_j$ are between $0$ and $\infty$ and $a_1 \ta \cdots \ta a_r = 0$, exactly as defined in \cite{DS04}.
But now, since at least one coefficient $a_j$ in the tropical linear combination is zero, and the last coordinate of each $x_j$ is zero, the result of the tropical linear combination is again with last entry being zero, which means that by our choice of mapping, the first $n$ coordinates of $\vary = a_1 \tm {\bf x_1} \ta \cdots \ta a_r {\bf x_r}$ are what we get if we do the computation in $\Rn$ with the same coefficients.

So, let us look more closely at the tropical linear combinations when done in $\Rn$.
When a tropical convex hull is of a finite number of generators then the result is a {\it tropical polytope} (see \cite{DS04}, \cite{MS15}).
For two generators $\varx=(x_1,\ldots,x_n)$ and $\vary=(y_1,\ldots,y_n)$, the convex hull is $\{ a \tm \varx \ta b \tm \vary :  a,b \geq 0 \mbox{ and } a \ta b = 0 \}$. When $a=b=0$ then $\varz = \varx \ta \vary = (\min(x_1,y_1),\ldots,\min(x_n,y_n))$.
If $a=0$ and $b$ goes over the positive real numbers then we obtain the tropical line segment from $\varz$ to $\varx$, and when $b=0$ and $a$ goes over the positive real numbers then we obtain the tropical line segment from $\varz$ to $\vary$. It follows that the set of tropical linear combinations of $\varx$ and $\vary$ in $\Rn$ is the tropical line segment between $\varx$ and $\vary$, which is indeed the tropical convex hull of $\{\varx, \vary\}$.

For more than two generators the convex hull $P$ may be composed of several polytopes of different dimensions (see \cite{DS04}, \cite{MS15}) and although $P$ is tropically convex, it is not, in general, a convex set in Euclidean geometry.
The next proposition is clear and follows from known results (see, e.g. \cite{MS15}, \cite{Yoshida21}).
\begin{proposition}
	 $\Bn$ is both convex and tropically convex.
\end{proposition}
\begin{proof}
	It is easy to see that if $\varx,\vary \in \Bn$ satisfy the two conditions of Observation~ \ref{obs:ball_conditions} then so does $\lambda \varx  + (1-\lambda) \vary$, for $0 \leq \lambda \leq 1$. Convexity follows also from the fact that $\Bn$ is the Minkowski sum of convex sets (Proposition~\ref{prop:Minkowski}). 
	 
	It is tropically convex since it is generated by a finite number of vectors, as shown in Proposition~\ref{prop:generators}.
	
	Another way: by Proposition~\ref{prop:geod_hull}, $\Bn$ is tropically geodesic and this implies convexity and tropical convexity. 
\end{proof}

Remember that we denoted by ${\bf \tilde{e}_{n+1}}$ the vector $-\One \in \Rn$.
\begin{proposition}
	\label{prop:generators}
	Let $S=\{-{\bf \tilde{e}_1}, \ldots, -{\bf \tilde{e}_{n+1}}\}$. Then $\Bn$ is generated by $S$, that is, $\Bn$ is the (min) tropical convex hull of $S$.
\end{proposition}
\begin{proof}
	Let $\varx = (x_1, \ldots, x_n) \in \Bn$ and assume, w.l.o.g., that
	$$
	-1 \leq x_1 \leq x_2 \leq \cdots \leq x_n \leq 1,
	$$
	with $x_n - x_1 \leq 1$.
	We need to show that ${\varx}$ equals a tropical linear combination of $-{\bf \tilde{e}_1}, \ldots, -{\bf \tilde{e}_{n+1}}$ with one coefficient being zero and the others non-negative.
	This indeed is the case:
	\begin{eqnarray*}
			(1+x_1)\tm(-{\bf \tilde{e}_1}) \hspace{-0.8em}&\ta&\hspace{-0.8em} (1+x_2)\tm(-{\bf \tilde{e}_2}) \ta \cdots \ta (1+x_n)\tm(-{\bf \tilde{e}_n}) \ta 0 \tm (-{\bf \tilde{e}_{n+1}}) \\
			&=&\hspace{-0.8em}(x_1,1+x_1,\ldots,1+x_1)\ta (1+x_2,x_2,1+x_2,\ldots,1+x_2)\ta  \\ &\cdots& \hspace{-0.5em} \ta (1+x_n,\ldots,1+x_n,x_n)\ta(1,1,\ldots,1) \\
			&=&\hspace{-0.8em} (x_1, \ldots, x_n)=\varx,
	\end{eqnarray*}
	because, for all $i \neq j$,
	$$
	x_i - x_j \leq x_n - x_1 \leq 1,
	$$
	that is,
	$$
	x_i \ta (1+x_j) = x_i.
	$$
	
	For the converse, we show that any vector generated by $-{\bf \tilde{e}_1}, \ldots, -{\bf \tilde{e}_{n+1}}$ is in $\Bn$. So, let $\varx = a_1 \tm (-{\bf \tilde{e}_1}) \ta \cdots \ta a_{n+1} \tm (-{\bf \tilde{e}_{n+1}})$, with $a_i \geq 0$, for all $i$, and $a_j=0$, for some $j$. The conditions on the coefficients and the fact that all entries of the generating vectors are between $-1$ and $1$ guarantee that this is also the case with the entries of $\varx$. In addition, since the difference between the maximal entry and the minimal entry in each $-{\bf \tilde{e}_i}$ is at most $1$, then this holds also for $\varx$ (regardless of the conditions imposed on the coefficients $a_i$). It follows that $\tnorm{\varx} \leq 1$, that is, $\varx \in \Bn$.
\end{proof}
	\begin{remarks}
		\begin{enumerate}
			\item The set $S=\{-{\bf \tilde{e}_1}, \ldots, -{\bf \tilde{e}_{n+1}}\}$ forms a {\it minimal} set of generators of $\Bn$ since it is easy to check that none of these vectors can be generated by the others. By \cite{DS04}, a minimal set of generators is {\it unique}.
			\item In the max-plus setting, where the identity element for tropical addition is $-\infty$, we would rather take non-positive coefficients in the tropical linear combination, with one coefficient being $0$, and then $\Bn$ is generated by ${\bf \tilde{e}_1}, \ldots, {\bf \tilde{e}_{n+1}}$.
		\end{enumerate}
	\end{remarks}

\section{Honeycomb of $\Rn$ with tropical balls}
\label{sec:honeycomb_Rn}
As is known, regular hexagons form a tessellation of the plane. The tropical disks of dimension 2 are elongated hexagons, so, they are not regular with respect to Euclidean metric, but they are regular in the tropical setting: the edges are of the same tropical length and the tropical angle between two (standard) adjacent edges is 2 radians (remember that we define a tropical circle to be of 6 radians and the tropical angle between two tropically orthogonal tropical lines is 2 radians). These hexagons also tile the plane (see Figure~\ref{fig:tessellation}). The centers of the tropical balls of radius 1 in the tropical honeycomb of the plane are at the points $\{ (c_1,c_2) \in\ZZ^2 : c_1+c_2 \equiv 0 \: (\mathrm{mod}~3) \}$, which are on the lattice generated by the (standard) vectors $(2,1)$ and $(-1,1)$.

This property of forming a honeycomb holds, in fact, in every dimension $n$, $n \geq 1$, that is, $n$-dimensional tropical balls $\Bn$ fill $\Rn$ facet-to-facet by translations.
This follows from $\Bn$ being defined by the vectors ${\bf \tilde{e}_1}, \ldots, {\bf \tilde{e}_{n}}, {\bf \tilde{e}_{n+1}}= -\One$. These vectors form the rows of a {\it totally unimodular matrix} (every square submatrix has determinant 0,1, or -1 ), or, equivalently, the associated matroid is regular \cite{Tut58}.
It follows from McMullen \cite{McM75} (see also \cite{DG04}) that such zonotopes form a honeycomb of $\Rn$. 
The honeycomb of tropical balls, when projected in $\RR^{n+1} / \RR\One$ to the hyperplane $\sum_{i=1}^{n+1}x_i = 0 \subset \RR^{n+1}$, is the Voronoi regions of the lattices $A_n$ that is described in the book of Conway and Slone \cite[Ch. 21]{CS88} (see also \cite{BBMP19}).
We give here an explicit proof of the tiling property of $\Bn$.
\begin{figure}[h]
	\centering
	\begin{tikzpicture}[scale=0.5, transform shape]
		\foreach \x in {0,...,6} {
			\foreach \y in {0,...,5} {
				\node at (\x+3*\y,-\x)
				{ \tikz\draw [thick,fill=blue!20, opacity=0.8]
					(1.0,0.0) -- (1.0,1.0) -- (0.0,1.0) -- (-1.0,0.0) -- (-1.0,-1.0) -- (0.0,-1.0) -- (1.0,0.0) ;}; }; }
	\end{tikzpicture}
	\caption{Tessellation of the plane by tropical disks}
	\label{fig:tessellation}
\end{figure}
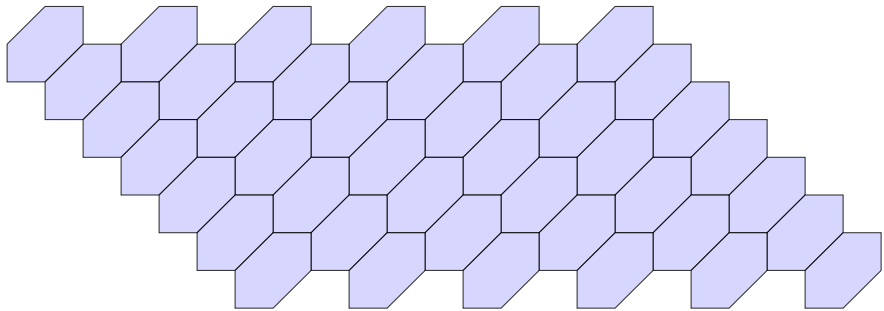

As we know, in dimension $n \geq 2$, standard balls do not form a honeycomb of $\Rn$ and the most we have is Besicovitch Covering Theorem \cite{Bes45}, which states that there is a disjoint almost covering of domains of finite regular Borel measure but with countably many balls of different sizes with zero infimum radius. \\

We denote by $\{a\}=a-\lfloor a \rfloor$ the fractional part of $a \in \RR$.
\begin{theorem}
	\label{thm:honeycomb}
	The collection of $n$-dimensional tropical unit balls $\Bn({\bf c})$ with centers at the lattice
	\begin{equation}
		\Lambda : \{ {\bf c} = (c_1,\ldots,c_n) \in\Zn : \sum_{i=1}^{n} c_i \equiv 0 \: (\mathrm{mod}~n+1)
		\label{eq:center} 
	\end{equation}
	forms a honeycomb of $\Rn$.
\end{theorem}
\begin{proof}
	Since the tropical unit balls $\Bn({\bf c})$ are closed sets, it suffices to show that every point in $\Rn$ which is not on the surface of a tropical unit ball belongs to exactly one ball.
	If ${\varx} = (x_1,\ldots,x_n) \in \Rn$ is an inner point of $\Bn({\bf c})$, with center ${\bf c} = (c_1,\ldots,c_n)$, then (at least) one of the following three conditions holds.
	\begin{enumerate}
		\item[(i)] $0 \leq x_i - c_i < 1$, for all $i$;
		\item[(ii)] $-1 < x_i - c_i \leq 0$, for all $i$;
		\item[(iii)] there exists a pair of indices $(i,j)$, such that $x_i > c_i$ and $x_j < c_j$, and for all such pairs $(i,j)$, $x_i+c_j-(x_j+c_i) < 1.$
	\end{enumerate}
	So, let ${\varx} = (x_1,\ldots,x_n) \in\Rn$, where ${\varx}$ is not on the surface of a tropical unit ball, and suppose, w.l.o.g., that
	$0 \leq \{x_1\}  \leq \{x_2\} \leq \cdots \leq \{x_n\}<1$.
	By the above conditions, if ${\varx} \in \Bn({\bf c})$ then,  for all $i$, $c_i=\lfloor x_i \rfloor$ or $c_i=\lfloor x_i \rfloor +1$.
	First, we show that ${\varx} \in \Bn({\bf c})$ for some $\bf c$ that satisfies $\eqref{eq:center}$.
	
	Let $\sum_{i=1}^{n} \lfloor x_i \rfloor \equiv k \: (\mathrm{mod}~n+1)$.
	\begin{enumerate}
		\item If $k=0$ then ${\varx}$ is an inner point of $\Bn({\bf c})$ with $c_i=\lfloor x_i \rfloor$, for all $i$, and satisfying (i).
		\item If $k=1$ then ${\varx}$ is an inner point of $\Bn({\bf c})$ with $c_i = \lfloor x_i \rfloor +1$, for all $i$, so that $\sum_{i=1}^{n} c_i \equiv 0 \: (\mathrm{mod}~n+1)$ and condition (ii) is satisfied. 
		\item If $1 < k \leq n$ then for ${\varx}$ to be an inner point of $\Bn({\bf c})$ we have that $n+1-k$ of the $c_i$ should have value $\lfloor x_i \rfloor +1$ and $k-1$ of them should have value $\lfloor x_i \rfloor$, so that $\sum_{i=1}^{n} c_i \equiv 0 \: (\mathrm{mod}~n+1)$ is satisfied. Moreover, condition (iii) is satisfied exactly when $c_i = \lfloor x_i \rfloor$, for $i=1,\ldots,k-1$, and $c_i = \lfloor x_i \rfloor + 1$, for $i=k,\ldots,n$, and, in addition, $\{x_{k-1}\} < \{x_k\}$ because otherwise $d_{\mathrm{tr}}({\bf c}, \varx) = \{x_{k-1}\}+1-\{x_k\} = 1$ and
		${\varx}$ is a boundary point.
	\end{enumerate}
	In all three cases, any other  number of indices $i$ for which $c_i$ are assigned the values $\lfloor x_i \rfloor$ and number of assignments $c_j=\lfloor x_j \rfloor +1$ results in $\sum_{i=1}^{n} c_i \not\equiv 0  \: (\mathrm{mod}~n+1)$. Therefore ${\varx}$ is an inner point of exactly one tropical unit ball.
\end{proof}

\bibliographystyle{plain}
\bibliography{Tropical_balls}
\end{document}